\documentclass[12pt,reqno]{amsart}
\usepackage{epsfig, amsmath, amssymb,latexsym}
\usepackage{hyperref}

\theoremstyle{plain}
 \makeatletter
\def\LaTeX{\leavevmode L\raise.42ex
\hbox{\kern-.3em\size{\sf@size}{0pt}\selectfont A}\kern-.15em\TeX}
\def\@currentlabel{2.1}\label{e:dispaa}
\def\@currentlabel{2.21}\label{e:dispau}
\def\@currentlabel{2.22}\label{e:dispav}
\def\@currentlabel{2.23}\label{e:dispaw}
\def\@currentlabel{2.24}\label{e:dispax}
\DeclareMathOperator{\im}{Im}
\DeclareMathOperator{\re}{Re}
\begin{document}

\title[Nonlocal Liouville Theorem]{A nonlinear Liouville theorem for fractional equations in the Heisenberg group}
\date{}

\author{Eleonora Cinti}

\address{E.C., Dipartimento di Matematica, Universit\`a degli Studi di Bologna,
Piazza di Porta San Donato 5,
40126 Bologna (Italy) }

\email{eleonora.cinti5@unibo.it }

\author{Jinggang Tan}

\address{J.T., Departamento de  Matem\'atica,
Universidad T\'{e}cnica Federico Santa Mar\'{i}a,  Avda. Espa\~na 1680,
Valpara\'{\i}so (Chile)}

\email{jinggang.tan@usm.cl}

\date{}\maketitle

\newcommand{\N}{\mathbb{N}}
\newcommand{\R}{\mathbb{R}}
\newcommand{\Z}{\mathbb{Z}}
\newcommand{\CH}{{\mathbb{H}}^{n}}
\newcommand{\bH}{\mathbb{H}}
\newcommand{\C}{\mathbb{C}}
\newcommand{\dC}{\mathcal{C}}
\newcommand{\dD}{\mathcal{V}}

\theoremstyle{plain}
\newtheorem{Thm}{Theorem}[section]
\newtheorem{Lem}[Thm]{Lemma}
\newtheorem{Def}[Thm]{Definition}
\newtheorem{Cor}[Thm]{Corollary}
\newtheorem{Prop}[Thm]{Proposition}
\newtheorem{Rem}[Thm]{Remark}
\newtheorem{Ex}[Thm]{Example}

\numberwithin{equation}{section}

\begin{abstract}
We establish a Liouville-type theorem for a subcritical nonlinear problem, involving a fractional power of the sub-Laplacian in the Heisenberg group. To prove our result we will use the local realization of fractional CR covariant operators, which can be constructed as the Dirichlet-to-Neumann operator of a degenerate elliptic equation in the spirit of Caffarelli and Silvestre \cite{CS}, as established in \cite{FGMT}. The main tools in our proof are the CR inversion and the moving plane method, applied to the solution of the lifted problem in the half-space $\CH\times \R^+$.
\end{abstract}
\section{Introduction and Main Results}
\setcounter{equation}{0}

In this paper we establish a Liouville-type result for the following fractional nonlinear problem in the Heisenberg group:
\begin{equation}\label{P1}
\mathcal P_{\frac{1}{2}} u=u^p\quad \mbox{in}\;\;\CH.
\end{equation}
Here $\mathcal P_{\frac{1}{2}}$ denotes a CR covariant operator of order $1/2$ in $\CH$, whose principal symbols agree with the pure fractional power $1/2$ of the Heisenberg Laplacian $-\Delta_{\mathbb H}$. In \cite{FGMT} Frank, Gonzalez, Monticelli and one of the authors, study  CR covariant operators of fractional orders on orientable and strictly pseudoconvex CR manifolds. In particular, they focuse on the construction of such operators as the Dirichlet-to-Neumann map associated to a degenarate elliptic equation in the spirit of Caffarelli and Silvestre \cite{CS}.

In this context, the Heisenberg group $\CH$ plays the same role as $\R^n$ in conformal geometry, in the sense that, as shown by Folland and Stein in \cite{FS}, $\CH$ approximates the pseudohermitian structure of a general orientable and strictly pseudoconvex CR manifolds (see also \cite{JL}).

Given a K\"ahler-Einstein manifold $\mathcal X$, CR covariant operators of fractional order $\gamma$ are pseudodifferential operators whose principal symbol agrees with the pure fractional powers of the CR sub-Laplacian on the boundary $\mathcal M=\partial \mathcal X$. They can be defined using scattering theory, as done in \cite{Epstein91, HislopPeterTang, Guillarmou08, Gover-Graham:CR-powers}. In the particular case of the Heisenberg group, they are the intertwining operators on the CR sphere, which can be calculated using representation theory techniques (see \cite{BFM}).


In \cite{FGMT}, in order to construct fractional CR covariant
 operators in the specific case of the Heisenberg group, $\CH$ is
 identified with the boundary of the Siegel domain in $\R^{2n+2}$
 (see Section 2 for the precise definition) and it is crucial to
 use its underlying complex hyperbolic structure.

Another possible approach in the construction of fractional powers of the sub-Laplacian consists in using purely functional analytic tools as done
by Ferrari and Franchi \cite{FF} who proved an extension result for fractional operators defined by using the spectral resolution of the sub-Laplacian in general Carnot groups, see also \cite{Fo75, Stein93}. The operators considered in \cite{FF} are different in nature from the ones in \cite{FGMT}, they correspond to the pure fractional powers of the sub-Laplacian and do not enjoy the CR covariance property.

\medskip
Since it will be of utmost importance in the sequel, we recall here the extension result proven in \cite{FGMT}.
\begin{Thm}[see Theorem 1.1 in \cite{FGMT}]\label{thm-extension}
Let $\gamma \in (0,1)$, $a=1-2\gamma$. For each $u \in C^\infty(\CH)$, there exists a unique solution $\mathcal E_\gamma u=U$ for the extension problem
\begin{equation}\label{extension}
\begin{cases}
\displaystyle \frac{\partial ^2 U}{\partial \lambda^2} +\frac{a}{\lambda}\frac{\partial U}{\partial\lambda} +\lambda^2 \frac{\partial^2U}{\partial t^2} +\frac{1}{2}\Delta_\bH U = 0 &\mbox{in}\:\: \widehat{\bH}^{n}_{+}:=\CH\times \R^+,\\
U=u &\mbox{on}\:\: \partial\widehat{\bH}^{n}_{+}=\CH\times \{\lambda=0\}.\\
\end{cases}
\end{equation}
Moreover,
$$\mathcal P_{\gamma} u=-c_\gamma \lim_{\lambda \rightarrow 0}\lambda^a\frac{\partial U}{\partial \lambda},$$
where $c_\gamma$ is a constant depending only on $\gamma$ which  precise value is given by
$$c_\gamma=\frac{\Gamma(\gamma)}{\gamma \Gamma(-\gamma)}\cdot 2^{2\gamma-1}.$$
\end{Thm}
In \eqref{extension} $\Delta_\bH$ denotes the sublaplacian in the Heisenberg group, which precise definition is given in Section 2 below.
Observe that, differently from the extension result established in \cite{FF}, here we have the additional term $\lambda^2 \frac{\partial^2U}{\partial t^2}$ which appears when one considers CR fractional sub-Laplacian.
When $a=1/2$ the equation in \eqref{extension} satisfied by $U$ becomes:
\begin{equation}\label{ext-1/2}
\frac{\partial ^2 U}{\partial \lambda^2}  +\lambda^2 \frac{\partial^2U}{\partial t^2} +\frac{1}{2}\Delta_\bH U = 0,
\end{equation}
and we have
$$\mathcal P_{\frac{1}{2}} u=-c_{\frac{1}{2}} \lim_{\lambda \rightarrow 0}\frac{\partial U}{\partial \lambda}.$$

Replacing  $\lambda$ in \eqref{ext-1/2} by $\sqrt{2}\lambda$, we will consider the operator
\begin{align}\label{oper-L}
{\mathcal{L}}
= \Delta_\bH + \frac{\partial^2}{\partial \lambda^2}+4\lambda^2\frac{\partial^2}{\partial t^2}.
\end{align}

Our Liouville-type theorem is the analogue, for the fractional operator $\mathcal P_{\frac{1}{2}}$, of a result by Birindelli and Prajapat \cite{BP}, for the sublaplacian $\Delta_\bH$. In \cite{BP}, the authors establish a nonexistence result for a class of positive solution of the equation
\begin{equation}\label{eq-laplacian}-\Delta_\bH u=u^p,
\end{equation}
for $p$ subcritical (i.e. $0<p<\frac{Q+2}{Q-2}$, where $Q=2n+2$ denotes the homogeneous dimension of $\CH$). The technique they used is based on the moving plane method (which goes back to Alexandrov \cite{Alex56} and Serrin  \cite{Serrin}), adapted to the Heisenberg group setting.
This method requires two basic tools: the maximum principle and invariance under reflection with respect to a hyperplane.
Since the operator $-\Delta_\bH$ is not invariant under the usual reflection with respect to hyperplanes, Birindelli
and Prajapat needed to introduce a new reflection,
 called $H$-reflection, under which $-\Delta_\bH$ is invariant.
 Since it will be important in the sequel, we recall here the definition of $H$-reflection.
\begin{Def}
For any $\xi=(x,y,t)\in \CH$, we consider the plane $T_\mu:=\{\xi\in \CH\;:\;t=\mu\}$. We define
$$\xi_\mu:=(y,x,2\mu-t),$$
to be the $H$-reflection of $\xi$ with respect to the plane $T_\mu$.
\end{Def}
Due to the use of this reflection, the proof of the non existence result in \cite{BP} requires the solution $u$ of \eqref{eq-laplacian} to be cylindrical, that is,  $u(x,y,t)=u(r_{0},t)$ must depend only on $r_{0}$ and $t$ where $r_{0}=(|x|^2+|y|^2)^{\frac{1}{2}}$.

We can now state our main result, which is the analogue for the operator $\mathcal P_{\frac{1}{2}}$ of the Liouville result contained in \cite{BP}.
\begin{Thm}\label{main}
Let $0<p<\frac{Q+1}{Q-1}$, where $Q=2n+2$ is the homogeneous dimension of $\CH$. Then there exists no cylindrical solution $u\in C^2(\CH)$ of
\begin{equation}
\begin{cases}
\mathcal P_{\frac{1}{2}}u=u^p &\mbox{in}\;\;\CH,\\
u>0 &\mbox{in}\;\;\CH.
\end{cases}
\end{equation}
\end{Thm}
Using the local formulation \eqref{extension} established in \cite{FGMT}, the above theorem will follow as a corollary of the following Liouville-type result for a nonlinear Neumann problem in the half-space $\CH\times \R^+$.
\begin{Thm}\label{main2}
Let $0<p<\frac{Q+1}{Q-1}$ and $U\in C^2(\CH\times\R^+)\cap C^1(\overline{\CH\times\R^+})$  be a nonnegative solution of
\begin{equation}\label{eq-ext}
\begin{cases}
\displaystyle\frac{\partial^2 U}{\partial \lambda^2}+4\lambda^2 \frac{\partial^2U}{\partial t^2}+\Delta_\bH U=0&\mbox{in}\:\:\CH\times \R^+,\\
\displaystyle -\frac{\partial U}{\partial \lambda}=U^p &\mbox{on}\:\:\CH\times \{\lambda=0\}.
\end{cases}
\end{equation}
Suppose that $U(x,y,t,\lambda)=U(r_{0},t,\lambda)$ depends only on $r_0,t, \lambda$, where $r_{0}=\left(|x|^2+|y|^2\right)^{\frac{1}{2}}$. Then $U\equiv 0$.
\end{Thm}

In the Euclidean case, classical nonexistence results for subcritical nonlinear problems in the all space $\R^n$
are contained in two works by Gidas and Spruck \cite{GS} and by Chen and Li \cite{CL}. Analogue results for nonlinear Neumann problems in the half-space $\R^n_+$ where established in \cite{LZ, LZhu}, using the methods of moving planes and moving spheres.

In the Heisenberg group setting there are several papers concerning nonexistence results for problem \eqref{eq-laplacian}.   Garofalo and Lanconelli \cite{GL} proved some nonexistence results for positive solutions of \eqref{eq-laplacian} when $p$ is subcritical, under some integrability conditions on $u$ and $\nabla u$. In \cite{LU, U} similar nonexistence results for positive solutions of \eqref{eq-laplacian} in the half-space are established for the critical exponent $p=\frac{Q+2}{Q-2}$. In \cite{BCCD}, a Liouville-type result for solution of \eqref{eq-laplacian} is proved without requiring any decay condition on $u$, but only for $0<p<\frac{Q}{Q-2}$. As explained before, in \cite{BP} Birindelli and Prajapat extends this last result to any $0<p<\frac{Q+2}{Q-2}$ but only in the class of cylindrical solution. A last more recent result in this context was proven by  Xu in \cite{LX}, who established that there are no positive solution of \eqref{eq-laplacian} for $0<p<\frac{Q(Q+2)}{(Q-1)^2}$. This result uses a different technique, based on the vector field method, and improves the results contained in \cite{GL} and \cite{BCCD}, since it does not require any decay on the solution $u$ and it improves the exponent $p$. Nevertheless it seems not allow  to reach the optimal exponent $\frac{Q+2}{Q-2}$ (observe that $\frac{Q}{Q-2}<\frac{Q(Q+2)}{(Q-1)^2}<\frac{Q+2}{Q-2}$).

In this paper we aim to establish a first Liouville-type result for a CR fractional power of $-\Delta_\bH$; this is, to our knowledge, the first nonexistence result in this fractional setting.

Let us comment now on the basic tools in the proof of our main result. Following \cite{BP}, in order to get a nonexistence result, we combine the method of moving planes with the CR inversion of the solution $u$.

The CR inversion was introduced by Jerison and Lee in \cite{JL}, and it is the analogue of the Kelvin transform, in the Heisenberg group context. In Section 3 we will give the precise definition of CR inversion and we will show which problem is satisfied by the CR inversion of a solution of \eqref{eq-ext}.

As said before, the moving plane method is based on several version of the maximum principles. More precisely we will recall the classical Bony's  maximum principle and we will prove two versions of the Hopf's Lemma (see Propositions \ref{hopf0} and \ref{hopf}).

The paper is organized as follows:
\begin{itemize}
\item in Section 2 we recall some basic facts on the Heisenberg group and we will introduce the fractional CR operator $\mathcal P_{\frac{1}{2}}$;
\item in Section 3 we will introduce the CR inversion of a function $u$ and prove  a lemma concerning the CR inversion of a solution of our problem \eqref{eq-ext};
\item in Section 4 we establish a maximum principle and Hopf's Lemma for our operator, which will be basic tools in the method of moving planes;
\item in Section 5 we will prove our main results Theorems \ref{main} and \ref{main2}.
\end{itemize}

\section{Preliminary facts on the Heisenberg group}
\setcounter{equation}{0}
In this section we recall some basic notions and properties concerning the Heisenberg group, see also \cite{BLU} and \cite{Stein93}.

We will denote the points in $\mathbb H^n$ using the notation $\xi=(x,y,t)=(x_1,...,x_n,y_1,...,y_n,t)\in \R^n\times \R^n\times \R$.
The Heisenberg group $\mathbb H^n$ is the space $\R^{2n+1}$ endowed with the group law $\circ$ defined in the following way:
$$  \hat{\xi} \circ \xi:=(  \hat{x}+x,  \hat{y}+y,  \hat{t}+ t +2\sum_{j=1}^n (x_j  \hat{y}_j -y_j \hat{x}_j)).$$
The natural dilation of the group is given by $\delta_{\ell}(\xi):=(\ell x, \ell y, \ell^2 t)$, and it satisfies $\delta_{\ell}(  \hat{\xi} \circ \xi)=\delta_{\ell}(  \hat{\xi})\circ \delta_{\ell}(\xi)$.

In $\mathbb H^n$ we will consider the gauge norm defined as
$$|\xi|_{\bH}:= \big[ \big(\sum_{i=1}^n (x_i^2+y_i^2)\big)^2 + t^2\big]^{\frac{1}{4}},$$
which is homogeneous of degree one with respect to $\delta_{\ell}$. Using this norm, one can define the distance between two points in the natural way:
$$d_{\bH}( \hat{\xi}, \xi)=|\hat{\xi}^{-1}\circ \xi|_{\bH},$$
where $\hat{ \xi}^{-1}$ denotes the inverse of $ \hat{\xi}$ with respect to the group action. We denote the ball associated to the gauge distance by
$$B_{\bH}(\xi_0,R):=\{\xi \in \bH^n\,:\,d_{\bH}(\xi,\xi_0)<R\}.$$
Denoting by $|A|$ the Lebesgue measure of the set $A$, we have that
$$|B_{\bH}(\xi_0,R)|=|B_{\bH}(0,R)|=R^Q|B_\bH(0,1)|.$$
Here $Q=2n+2$ denotes the homogeneous dimension of $\bH^n$.

For every $j=1,\cdots,n$, we denote by $X_j$, $Y_j$, and $T$ the following vector fields:
$$X_j=\frac{\partial}{\partial x_j} + 2y_j \frac{\partial}{\partial t},\quad Y_j=\frac{\partial}{\partial y_j} - 2x_j \frac{\partial}{\partial t}, \quad T=\frac{\partial}{\partial t}.$$
They form a basis of the Lie Algebra of left invariant vector fields. Moreover, an easy computions shows that $[X_k,Y_j]=-4\delta_{kj}T$. The Heisenberg gradient of a function $f$ is given by $$\nabla_{\bH}f=(X_1f,\cdots,X_n f,Y_1 f,\cdots,Y_n f).$$ Finally, we define the sublaplacian as
\begin{eqnarray*}
&&\Delta_\bH:=\sum_{j=1}^n (X_j^2+Y_j^2)\\
&&\hspace{1em} =\sum_{j=1}^n \frac{\partial^2}{\partial x_j^2} +\frac{\partial^2}{\partial y_j^2} + 4 y_j \frac{\partial^2}{\partial x_j \partial t} - 4 x_j \frac{\partial^2}{\partial y_j \partial t} + 4(x_j^2+y_j^2) \frac{\partial^2}{\partial t^2}.
\end{eqnarray*}
It can be written also in the form $\Delta_\bH=\mbox{div}(  \overline{A}\nabla^T)$, where $\overline A=\overline{a}_{kj}$ is the $(2n+1)\times (2n+1)$ symmetric matrix given by $\overline{a}_{kj}=\delta_{kj}$ for $k,j=1,...,2n$, $\overline{a}_{j(2n+1)}=\overline{a}_{(2n+1)j}=2y_j$ for $j=1,...,n$, $\overline{a}_{j(2n+1)}=\overline{a}_{(2n+1)j}=-2x_j$ for $j=n+1,...,2n$ and $\overline a_{(2n+1)(2n+1)}=4(|x|^2+|y|^2)$. It is easy to observe that $\overline{A}$ is positive semidefinite for any $(x,y,t)\in \mathbb{H}^n$.
This operator is degenerate elliptic, and it is hypoelliptic since it satisfies the H\"ormander condition.

\medskip

We pass now to describe CR covariant operators of fractional orders in $\CH$. For more precise notions of CR geometry and for the construction of CR covariant fractional powers of the sub-Laplacian on more general CR manifolds, we refer to \cite{FGMT} and references therein. Here we just consider the case of the Heisenberg group, since it is the one of interest.

Introducing complex coordinates $\zeta=x+iy \in \C^n$, we can identify the Heisenberg group $\CH$ with the boundary of the Siegel domain $\Omega_{n+1}\subset \mathbb C^{n+1}$, which is  given by
$$\Omega_{n+1}:=\left\{(\zeta_1,\ldots,\zeta_{n+1})=(\zeta,\zeta_{n+1})\in\C^{n}\times \C
\,|\,q(\zeta,\zeta_{n+1})>0\right\},$$
with
$$q(\zeta,\zeta_{n+1})=\im \zeta_{n+1}-\sum_{j=1}^n|\zeta_j|^2,$$
through the map $(\zeta,t)\in \CH \rightarrow (\zeta,t+i|\zeta|^2)\in \partial \Omega_{n+1}$.
It is possible to see that $\mathcal X=\Omega_{n+1}$ is a K\"ahler-Einstein manifold, endowed with a K\"ahler form $\omega_+$ (which precise expression can be found in formula (1.4) in \cite{FGMT}), and a corresponding K\"ahler metric $g_+$ (see formula (1.6) in \cite{FGMT}). Using this metric, one can see that  $\Omega_{n+1}$ can be identified with the complex hyperbolic space. The boundary manifold $\mathcal M=\partial \Omega_{n+1}$ inherits a natural CR structure from the complex structure of the ambient manifold. Given a CR structure, it is possible to associate to it a contact form $\theta$, that in the specific case of the Heisenberg group, is given by %
\begin{equation}\label{theta}
\theta=\left[ dt + \frac{1}{2}\sum_{j=1}^n (x_j dy_j - y_j dx_j) \right].
\end{equation}
This form satisfies $\theta(T)=1$.

Scattering theory tells us that for $s\in\mathbb C$, $\re(s)> \frac{m}{2}$, and except for a set of exceptional values, given $f$ smooth on $\mathcal M$, the eigenvalue equation
\begin{equation*}
-\Delta_{g^+} u -s(m-s)u=0,\quad\text{in }\mathcal X
\end{equation*}
has a solution $u$ with the expansion
\begin{equation*}
\begin{cases}
u=q^{(m-s)}F+q^{s} G & \text{for some}\quad F,G\in C^\infty(\overline{\mathcal X}),\\
 F|_{\mathcal M}=f.
\end{cases}\end{equation*}
The scattering operator is defined as
$$S(s): \mathcal C^\infty(\mathcal M) \to \mathcal C^\infty(\mathcal M)$$
by
$$S(s)f:=G|_{\mathcal M}.$$

 We set $s=\frac{m+\gamma}{2}$, for $\gamma\in(0,m)\backslash\mathbb N$. The conformal fractional sub-Laplacian on $\CH=\mathcal M$ (associated to the contact form $\theta$) is defined in the following way:
\begin{equation}\label{operators}\mathcal{P}^\theta_\gamma f=C_\gamma S(s)f,\end{equation}
for a constant
$$C_\gamma=2^{2\gamma-1}\frac{\Gamma(\gamma)}{\gamma\Gamma(-\gamma)}.$$
%
For $\gamma=1$ and $\gamma=2$ we have (see \cite{FGMT}):
$$P^\theta_1 =-\Delta_{\mathbb H}\quad \mbox{and}\quad P^\theta_2 =\Delta^2_{\mathbb H}+T^2.$$
A crucial property of $\mathcal{P}^\theta_\gamma$ is its conformal covariance. Indeed, if we consider a conformal change of the contact form $\hat \theta= w^{\frac{2}{n+1-\gamma}}\theta$, then the corresponding fractional operator is given by:
$$\mathcal P_\gamma^{\hat\theta}(\cdot)= w^{-\frac{n+1+\gamma}{n+1-\gamma}}\mathcal P^\theta_\gamma (w\:\:\cdot).$$
For further details on CR covariance and on the geometric properties of the operator $\mathcal P_\gamma^\theta$ we refer to \cite{FGMT}; here we just emphasize that this covariance property is reflected in the fact that the extension operator $\mathcal L$ defined in \eqref{oper-L} well behaves under $CR$ inversion (as we will see later in Section 3), and this will be crucial in the proof of our main result.
\medskip


As explained in the introduction, one of the main result in \cite{FGMT},
is the characterization of these fractional operators via
the extension problem \eqref{extension}. Since throughout this paper
we will work on this lifted problem in the extended space, let us introduce some notations in $\widehat{\bH}^{n}={\bH}^{n}\times \R^+$.

Since, the contact form $\theta$ is fixed (and it is the one defined in \eqref{theta}), for simplicity of notations we will write $\mathcal P_\gamma$ instead of $\mathcal P^\theta_\gamma$.

Analougsly to $\bH^n$, in $\widehat{\bH}^{n}$ we define the following group low (that for simplicity of notation we still denote by $\circ$):\\
for $z=(x_1,\cdots,x_n,y_1,\cdots,y_n,t,\lambda)\in \widehat{\bH}^{n}$ and
$\hat{z}=(\hat{x}_1,\cdots, \hat{x}_n, \hat{y}_1,\cdots,
\hat{y}_n,  \hat{t}, \hat{\lambda})\in \widehat{\bH}^{n}$, we set
$$\hat z \circ z:=(\hat x+x,\hat y + y, \hat t + t +2\sum_{j=1}^n(x_j \hat y_j - y_j\hat x_j), \hat\lambda + \lambda).$$
Moreover we consider  the norm given by
$$|z|_{\widehat{\bH}^{n}}:=[(|x|^2+|y|^2+\lambda^2)^2+ t^2]^{\frac{1}{4}}.$$
Finally we  denote  the distance   $d_{\widehat{\bH}}$ between
  $z$ and $\hat{z}$,  by
$$d_{\widehat{\bH}}(z, \hat{z}):=|\hat z^{-1} \circ z|_{\widehat{\bH}^{n}}.$$
Observe that when $\lambda=\hat\lambda=0$, that is $z$ and $\hat z$ belong to $\CH$,
$d_{\widehat{\bH}}(z, \hat{z})=d_{\bH}(z, \hat{z})$.
Moreover, given $\bar z \in   \widehat{\bH}^{n}$ we set
$${\mathcal{B}}(\bar z,R)=\{z\in \mathbb C^{n+1}\mid d_{\widehat{\bH}}(z,\bar z)<R\}$$
and for any $z_0\in \CH\times \{0\}$ we denote
$${\mathcal{B}}^+(z_0,R)=\{z\in \CH\times \R^+\mid d_{\widehat{\bH}}(z,z_0)<R\,,\lambda>0\}.$$

The operator  $\mathcal L $ (defined in \eqref{oper-L}),
  writing explicitly all the terms, becomes
\begin{align*}
{\mathcal{L}}&=\frac{\partial^{2}}{\partial \lambda^{2}}+\sum_{j=1}^n\left(\frac{\partial^{2}}{\partial x_{j}^{2}}+\frac{\partial^{2}}{\partial y_{j}^{2}}
+4y_{j}\frac{\partial^{2}}{\partial x_{j}\partial t }-4x_{j}\frac{\partial^{2}}{\partial y_{j}\partial t }\right)\\
&\hspace{1em}+
4(\lambda^{2}+\sum_{j=1}^n(x_{j}^{2}+y_{j}^{2}))\frac{\partial^{2}}{\partial t^{2}}.\end{align*}
Also in this case, we can write $\mathcal{L}=\mbox{div}(A\nabla^T)$, where now
 $A$ is the $(2n+2)\times (2n+2)$ symmetric matrix given by $a_{kj}=\delta_{kj}$ if $k,j=1,\cdots, 2n$,
$a_{j(2n+1)}=a_{(2n+1)j}=2y_{j}$  if $j=1,\cdots, n$, $a_{j(2n+1)}=a_{(2n+1)j}=-2x_{j}$  if $j=n+1,\cdots, 2n$,
$a_{(2n+1)(2n+1)}=4(|x|^{2}+|y|^2+\lambda^2)$,  $a_{(2n+2)(2n+2)}=1$, $a_{j(2n+2)}= a_{(2n+2)j}=0$ if $j=1,\cdots, 2n+1$.

In the sequel it will be useful to express $\mathcal L$ for cylindrical and radial functions.

For a point $z=(x,y,t,\lambda)\in {\widehat{\bH}^n}$, let
\begin{align*}
r&=(|x|^{2}+|y|^{2}+\lambda^{2})^{1/2},\\
\rho&=(r^{4}+t^{2})^{1/4}.
\end{align*}
 Suppose that $\Psi$ is a radial function, that is, $\Psi$ depends only on $\rho$; then
a direct computation gives:
 \begin{align*}
&\frac{\partial \Psi}{\partial x_{j}}=\frac{\partial \Psi}{\partial \rho}\frac{\partial \rho}{\partial x_{j}}=
\rho^{-3}r^{2}x_{j}\frac{\partial \Psi}{\partial \rho}, \;\;\;
\frac{\partial \Psi}{\partial y_{j}}=\frac{\partial \Psi}{\partial \rho}\frac{\partial \rho}{\partial y_{j}}=
\rho^{-3}r^{2}y_{j}\frac{\partial \Psi}{\partial \rho},\\
&\frac{\partial \Psi}{\partial \lambda}=\frac{\partial \Psi}{\partial \rho}\frac{\partial \rho}{\partial \lambda}=
\rho^{-3}r^{2}\lambda\frac{\partial \Psi}{\partial \rho}.
\end{align*}
Then we deduce that
\begin{align*}
&\frac{\partial^{2} \Psi}{\partial x_{j}^{2}}=
\frac{ r^{4}x_{j}^{2}}{\rho^{6}}\frac{\partial^{2} \Psi}{\partial \rho^{2}}+\frac{\rho^{4}r^{2}+2\rho^{4}x_{j}^{2}-3r^{4}x_{j}^{2}}{\rho^{7}}\frac{\partial \Psi}{\partial \rho}, \\
&\frac{\partial^{2} \Psi}{\partial y_{j}^{2}}=
\frac{ r^{4}y_{j}^{2}}{\rho^{6}}\frac{\partial^{2} \Psi}{\partial \rho^{2}}+\frac{\rho^{4}r^{2}+2\rho^{4}y_{j}^{2}-3r^{4}y_{j}^{2}}{\rho^{7}}\frac{\partial \Psi}{\partial \rho},\\
&\frac{\partial^{2} \Psi}{\partial \lambda^{2}}=
\frac{ r^{4}\lambda^{2}}{\rho^{6}}\frac{\partial^{2} \Psi}{\partial \rho^{2}}+\frac{\rho^{4}r^{2}+2\rho^{4}\lambda^{2}-3r^{4}\lambda^{2}}{\rho^{7}}\frac{\partial \Psi}{\partial \rho},\\
&\frac{\partial^{2} \Psi}{\partial t^{2}}=
\frac{ t^{2}}{4\rho^{6}}\frac{\partial^{2} \Psi}{\partial \rho^{2}}+\frac{2\rho^{4}-3t^{2}}{4\rho^{7}}\frac{\partial \Psi}{\partial \rho}.
\end{align*}
Hence, by using that
\begin{align*}
\frac{1 }{\rho^{6}}( \sum_{j=1}^n(r^{4}x_{j}^{2}+r^{4}y_{j}^{2})+r^{4}\lambda^{2}+r^{2}t^{2})=\frac{r^{2}}{\rho^{2}}
\end{align*}
and
\begin{align*}
&\frac{1}{\rho^{7}}\big[(2n+1)\rho^{4}r^{2}+\big(2\rho^{4}
 -3r^{4}+  (2\rho^{4}-3t^{2})\big) (\sum_{j=1}^{n}(x_{j}^{2}+ y_{j}^{2})
+\lambda^{2})\big]
\\
=&\frac{1}{\rho^{7}}[(2n+1)\rho^{4}r^{2}+2\rho^{4}r^{2}
-3r^{6}+ r^{2}(2\rho^{4}-3t^{2})]
=\frac{Q r^{2}}{\rho^{3}},
\end{align*}
we conclude that
\begin{align}\label{radial}
{\mathcal{L}}\Psi(\rho)=\frac{r^2}{\rho^{2}}\left(\frac{d^{2} \Psi(\rho)}{d\rho^{2}}+
\frac{Q}{ \rho}\frac{d \Psi(\rho)}{d \rho}\right).
\end{align}

In a similar way, we deduce that for a cylindrical symmetric function $\phi=\phi(r,t)$,
\begin{align}
{\mathcal{L}}\phi=\frac{\partial^{2}\phi}{\partial r^{2}}+\frac{Q-2}{r}\frac{\partial\phi}{\partial r}+
4r^{2}\frac{\partial^{2}\phi}{\partial t^{2}}.
\end{align}
Using the radial form \eqref{radial} for $\mathcal L$, an easy computation yields the following lemma.
\begin{Lem}\label{foundamental}
Let
$\psi(\rho)=\frac{1}{\rho^{Q-1}}=\frac{1}{\rho^{2n+1}}$ for $\rho\neq 0$.
 Then we have that
\begin{equation}
\begin{cases}
\mathcal L\psi(\rho)=0 &\text{in} \:\:\CH\times \R^+\setminus \{0\},\\
-\frac{\partial}{\partial \lambda}\psi(\rho)=0 &\mbox{on}\:\:\CH \times \{\lambda=0\}\setminus \{0\}.
\end{cases}
\end{equation}
\end{Lem}


\section{CR inversion}
\setcounter{equation}{0}

Following \cite{BP} and \cite{JL}, we define the CR inversion in the half-space $\widehat{\bH}^{n}_{+}=\CH\times \R^+$.

For any $(x,y,t,\lambda)\in \widehat{\bH}^{n}_{+}$, let as before $r=(|x|^2+|y|^2+\lambda^2)^{\frac{1}{2}}$ and $\rho=(r^4+t^2)^{\frac{1}{4}}$. We set
\begin{align*}
\widetilde x_i=\frac{x_i t+y_i r^{2}}{\rho^4},\;\;\widetilde y_i=\frac{y_i t-x_i r^{2}}{\rho^4},\;\;\widetilde t=-\frac{t}{\rho^4},\;\;\widetilde \lambda=\frac{\lambda}{\rho^2}.
\end{align*}
The CR inversion of a function $U$ defined on $\widehat{\bH}^{n}_{+}$, is given by
$$v(x,y,t,\lambda)=\frac{1}{\rho^{Q-1}} U(\tilde{x},\tilde{y},\tilde{t},\tilde{\lambda}).$$
The following lemma shows which equation is satisfied by the CR inversion of a solution of problem \eqref{eq-ext}.
\begin{Lem}\label{CR-eq}
Suppose that $U\in C^2(\widehat{\bH}^{n}_{+})\cap C(\overline{\widehat{\bH}^{n}_{+}})$ is a solution of \eqref{eq-ext}. Then the CR inversion $v$ of $U$ satisfies
\begin{equation}
\begin{cases}
\mathcal L v=0 &\mbox{in}\;\;\widehat{\bH}^{n}_{+} \{0\},\\
\displaystyle -\frac{\partial v}{\partial \lambda}= \rho^{p(Q-1)-(Q+1)} v^p &\mbox{on}\;\;\CH\times \{\lambda=0\}\setminus \{0\}.
\end{cases}
\end{equation}
\end{Lem}

%
\begin{proof}
Since it is just a long computation, we will give the details of the proof for cylindrical solutions $U(r,t)$, but the statement holds true for any solution.

It is clear that
\begin{align*}
\tilde{r}^{2}=|\tilde{x}|^{2}+|\tilde{y}|^{2}+\tilde{\lambda}^{2}=
\frac{t^{2}(|x|^{2}+|y|^{2}+\lambda^{2})}{\rho^{8}}+\frac{r^{4}(|x|^{2}+|y|^{2}+\lambda^{2})}{\rho^{8}}=
\frac{r^{2}}{\rho^{4}},
\end{align*}
and
$$\tilde \rho=(|\tilde{r}|^{4} + \tilde t^2)^{\frac{1}{2}}=\frac{1}{\rho}.$$
In the same way:
$$\tilde{r_0}^{2}=|\tilde{x}|^{2}+|\tilde{y}|^{2}=\frac{r_0^{2}}{\rho^{4}}.$$
Therefore, if $U$ depends only on $(r,t)$ (respectively on $(r_0,t,\lambda)$), then so does also $v$.
The following relations are useful
\begin{align*}
&\frac{\partial \tilde{r}}{\partial r}=\frac{t^{2}-r^{4}}{\rho^{6}},\; \frac{\partial \tilde{r}}{\partial t}
=\frac{-rt}{\rho^{6}},\\
&\frac{\partial \tilde{t}}{\partial r}=\frac{4tr^{3}}{\rho^{8}},\; \frac{\partial \tilde{t}}{\partial t}
=\frac{2t^{2}-t^{2}-r^{4}}{\rho^{8}}=\frac{t^{2}-r^{4}}{\rho^{8}}.
\end{align*}
Now we have
\begin{align*}
\frac{\partial  v}{\partial r}=\frac{(1-Q)r^{3}}{\rho^{Q+3}}U+\frac{1}{\rho^{Q-1}}
[\frac{t^{2}-r^{4}}{\rho^{6}}\frac{\partial  U}{\partial \tilde{r}}
+ \frac{4tr^{3}}{\rho^{8}}\frac{\partial  U}{\partial \tilde{t}}]
\end{align*}
and
\begin{align*}
\frac{\partial^{2} v}{\partial r^{2}}&=
\frac{\partial  }{\partial r}\big(\frac{(1-Q)r^{3}}{\rho^{Q+3}}\big) U
+\frac{2(1-Q)r^{3}}{\rho^{Q+3}}[(\frac{t^{2}-r^{4}}{\rho^{6}})\frac{\partial U}{\partial \tilde{r}}+
(\frac{4r^{3}t}{\rho^{8}})\frac{\partial U}{\partial \tilde{t}}]\\
&+\frac{1}{\rho^{Q-1}}[\frac{\partial  }{\partial r}(\frac{t^{2}-r^{4}}{\rho^{6}})\frac{\partial U}{\partial \tilde{r}}+\frac{\partial  }{\partial r}(\frac{4r^{3}t}{\rho^{8}})\frac{\partial U}{\partial \tilde{t}}]\\
&+
\frac{1}{\rho^{Q-1}}\Big[2(\frac{t^{2}-r^{4}}{\rho^{6}})^{2}\frac{\partial^{2} U}{\partial \tilde{r}^{2}}+
(\frac{t^{2}-r^{4}}{\rho^{6}})(\frac{4r^{3}t}{\rho^{8}})\frac{\partial^{2} U}{\partial \tilde{r}\partial \tilde{t}}+(\frac{4r^{3}t}{\rho^{8}})^{2}\frac{\partial U}{\partial \tilde{t}^{2}}\Big].
\end{align*}
The derivatives with respect to $t$ are given by
\begin{align*}
\frac{\partial  v}{\partial t}=\frac{(1-Q)t}{2\rho^{Q+3}}U+\frac{1}{\rho^{Q-1}}
[\frac{-rt}{\rho^{6}}\frac{\partial  U}{\partial \tilde{r}}
+ \frac{t^{2}-r^{4}}{\rho^{8}}\frac{\partial  U}{\partial \tilde{t}}]
\end{align*}
and
\begin{align*}
\frac{\partial^{2} v}{\partial t^{2}}&=\frac{\partial  }{\partial t}\big(\frac{(1-Q)t}{2\rho^{Q+3}}\big) U
+\frac{2(1-Q)t}{2\rho^{Q+3}}[(\frac{-rt}{\rho^{6}})\frac{\partial U}{\partial \tilde{t}}+
(\frac{t^{2}-r^{4}}{\rho^{8}})\frac{\partial U}{\partial \tilde{t}}]\\
&+\frac{1}{\rho^{Q-1}}[\frac{\partial  }{\partial t}(\frac{-rt}{\rho^{6}})\frac{\partial U}{\partial \tilde{r}}+\frac{\partial  }{\partial r}(\frac{t^{2}-r^{4}}{\rho^{8}})\frac{\partial U}{\partial \tilde{t}}]\\
&+
\frac{1}{\rho^{Q-1}}\Big[(\frac{-rt}{\rho^{6}})^{2}\frac{\partial^{2} U}{\partial \tilde{r}^{2}}+
(\frac{-rt}{\rho^{6}})(\frac{t^{2}-r^{4}}{\rho^{8}})\frac{\partial^{2} U}{\partial \tilde{r}\partial \tilde{t}}+(\frac{t^{2}-r^{4}}{\rho^{8}})^{2}\frac{\partial U}{\partial \tilde{t}^{2}}\Big]\\
&+\frac{1}{\rho^{Q-2}}[\frac{\partial }{ \partial t}2(\frac{-rt}{\rho^{6}})
\frac{\partial U}{ \partial \tilde{r}}+ \frac{\partial }{ \partial t}(\frac{t^{2}-r^{4}}{\rho^{6}})
\frac{\partial U}{ \partial \tilde{t}}].
\end{align*}

Let us denote
\begin{align*}
{\mathcal{L}}U=a_{1}\frac{\partial^{2}U}{\partial \tilde{r}^{2}}+a_{2} \frac{\partial^{2}U}{\partial \tilde{r}\partial \tilde{t}}+
a_{3}\frac{\partial^{2}U}{\partial \tilde{t}^{2}}+b_{1} \frac{\partial U}{\partial \tilde{r}}+
b_{2} \frac{\partial U}{\partial \tilde{t}} +cU.
\end{align*}
Then
\begin{align*}
&c=0,\\
&a_{1}=\frac{1}{\rho^{Q+11}}[(t^{2}-r^{4})^{2}+4r^{2}(rt)^{2}]=\frac{1}{\rho^{Q+3}},\\
&a_{2}=\frac{1}{\rho^{Q+13}}[2(t^{2}-r^{4})(4r^{3}t)-8r^{2}(rt)(t^{2}-r^{4})]=0,\\
&a_{3}=\frac{1}{\rho^{Q+15}}[16r^{6}t^{2}+4r^{2}(t^{2}-r^{4})^{2}]=\frac{1}{\rho^{Q+3}}\frac{4r^{2}}{\rho^{4}}.
\end{align*}
By using that
\begin{align*}
&\frac{\partial}{\partial
r}(\frac{t^{2}-r^{4}}{\rho^{6}})
=\frac{-4r^{3}(t^{2}+r^{4})}{\rho^{10}}-\frac{6(t^{2}-r^{4})r^{3}}{\rho^{10}},
\\
&\frac{\partial}{\partial
t}(\frac{-rt}{\rho^{6}})=-\frac{r(t^{2}+r^{4})}{\rho^{10}}+\frac{3rt^{2}}{\rho^{10}},
\end{align*}
we can see that
\begin{align*}
 b_{1}=&\frac{2(1-Q)r^{4}(t^{2}-r^{4})}{r\rho^{Q+9}}
+\frac{1}{\rho^{Q-1}}\frac{\partial}{\partial r}(\frac{t^{2}-r^{4}}{\rho^{6}})\\
&
+\frac{Q-2}{r}\frac{(t^{2}-r^{4})(t^{2}+r^{4})}{\rho^{Q+9}}+\frac{4r^{2}(1-Q)(-r^{2}t^{2})}{r\rho^{Q+9}}
+\frac{4r^{2}}{\rho^{Q-1}}
\frac{\partial}{\partial t}(\frac{-rt}{\rho^{6}})
\\
=&\frac{2(1-Q)r^{4}(t^{2}-r^{4})}{r\rho^{Q+9}}
+ (\frac{-4r^{4}(t^{2}+r^{4})}{r\rho^{Q+9}}-\frac{6(t^{2}-r^{4})r^{4}}{r\rho^{Q+9}})\\
&
+\frac{Q-2}{r}\frac{(t^{2}-r^{4})(t^{2}+r^{4})}{\rho^{Q+9}}+\frac{4(1-Q)(-r^{4}t^{2})}{r\rho^{Q+9}}\\
&-\frac{4r^{4}(t^{2}+r^{4})}{r\rho^{Q+9}}+\frac{12r^{4}t^{2}}{r\rho^{Q+9}}
\\
=&\frac{1}{\rho^{Q+3}}\frac{(Q-2)\rho^{2}}{r}=\frac{1}{\rho^{Q+3}}\frac{(Q-2)}{\tilde{r}}.
\end{align*}

We have that  \begin{align*}
&\frac{\partial}{\partial
r}(\frac{4r^{3}t}{\rho^{8}})=
\frac{12r^{2}t(t^{2}+r^{4})}{\rho^{12}}-\frac{32r^{6}t}{\rho^{12}},\\
&\frac{\partial}{\partial t}(\frac{ t^{2}-r^{4}}{\rho^{8}})=
\frac{2t(t^{2}+r^{4})}{\rho^{12}}-\frac{4t(t^{2}-r^{4})}{\rho^{12}}.
\end{align*}
Then we see that the following term vanishes, that is,
\begin{align*}
b_{2}=&\frac{8(1-Q)r^{6}t}{\rho^{Q+11}}+\frac{1}{\rho^{Q-1}}
\frac{\partial}{\partial r}(\frac{4r^{3}t}{\rho^{8}})
+ \frac{4(Q-2)r^{3}t(t^{2}+r^{4})}{r\rho^{Q+11}}\\
&+\frac{4(1-Q)r^{2}t(t^{2}-r^{4})}{\rho^{Q+11}}
+\frac{4r^{2}}{\rho^{Q-1}}\frac{\partial}{\partial t}(\frac{ t^{2}-r^{4}}{\rho^{8}})\\
=&\frac{8(1-Q)r^{6}t}{\rho^{Q+11}}+
(\frac{12r^{2}t(t^{2}+r^{4})}{\rho^{Q+11}}-\frac{32r^{6}t}{\rho^{Q+11}})
+ \frac{4(Q-2)r^{2}t(t^{2}+r^{4})}{\rho^{Q+11}}\\
&+\frac{4(1-Q)r^{2}t(t^{2}-r^{4})}{\rho^{Q+11}}
+ (\frac{8r^{2}t(t^{2}+r^{4})}{\rho^{Q+11}}-\frac{16  r^{2}t(t^{2}-r^{4})}{\rho^{Q+11}})=0.
\end{align*}
This shows that $\mathcal L v(x,y,t,\lambda)=\frac{1}{\rho^{(Q+3)}}\mathcal L U(\widetilde x,\widetilde y,\widetilde t,\widetilde \lambda)=0$. Finally, the Neumann data becomes:
\begin{align*}
-\lim_{\lambda \rightarrow 0} \frac{\partial v}{\partial \lambda}(x,y,t,\lambda)=&-\lim_{\widetilde \lambda \rightarrow 0}\frac{1}{\rho^{Q-1}}\rho^{-2} \frac{\partial U}{\partial \widetilde \lambda}(\widetilde x,\widetilde y,\widetilde t,\widetilde \lambda)\\
 =&\frac{1}{\rho^{(Q+1)}} U^p(\widetilde x,\widetilde y,\widetilde t,0)=\rho^{p(Q-1)-(Q+1)} v(x,y,t,0).
\end{align*}
\end{proof}

\section{Maximum principle and Hopf's Lemma}
\setcounter{equation}{0}

Basic tools in the method of moving planes are the maximum principle and Hopf's Lemma.
We start by recalling the classical maximum principle for H\"ormander-type operators due to Bony \cite{Bony}.

\medskip
\begin{Prop}(\cite{Bony}) \label{prop-max0}
Let $\dD$ be a bounded domain in ${\widehat{\bH}}^{n}$,  let $Z$ be
a smooth vector field on $\dD$ and let $a$ be a smooth nonnegative function. Assume that $U\in  C^{2}(\dD)\cap C^{1}(\overline{\dD})$ is a solution of
\begin{equation}\label{MP0}
\begin{cases}
-\mathcal{L}U + Z(z)U+ a(z) U\ge 0 & \mbox{in}\;\;\dD,\\
U\ge 0 &\mbox{on} \;\; \partial\dD.\\
\end{cases}
\end{equation}
Then $U\geq 0$ in $\dD$.
\end{Prop}
We prove now two Hopf's Lemmas. The first one is Hopf's Lemma for the operator
$\mathcal L$ in a subset $\mathcal V$ of $\widehat{\bH}^{n}$. We first define the interior ball condition in this setting.

\begin{Def}
Let $\dD\subset \widehat{\bH}^{n}$. We say that $\dD$ satisfies the interior $d_{\widehat{\bH}}$-ball condition at $P \in \partial \dD$ if there exist a constant $R>0$ and a point $z_0 \in \dD$, such that the ball $\mathcal{B}(z_0,R)\subset \dD$ and $P \in \partial \mathcal{B}(z_0,R)$,
where $\mathcal{B}(z_{0},R)=\{z\in {\widehat{\bH}}^{n}\mid d_{\widehat{\bH}}(z,z_{0})<R\}$.
\end{Def}

\begin{Lem}\label{hopf0}
Let $\dD\subset  {\widehat{\bH}}^{n}$ satisfy the interior $d_{\widehat{\bH}}$-ball condition at the point $P_0\in \partial \dD$ and
let $U\in C^2(\dD)\cap C^1(\overline{\dD})$, be a solution of
\begin{equation}\label{eq-hopf0}
-\mathcal L U \geq c_{1}(z)U\; \text{in}\;\;\dD
\end{equation}
with $c_{1}\in L^\infty(\dD)$. Suppose that $U(z)>U(P_0)=0$ for every $z\in \dD$.

Then
\[
 \lim\limits_{s\rightarrow 0}\frac{U(P_{0})-U(P_{0}-s\nu)}{s}<0.
\]
where $\nu$ is the outer normal to $\partial \dD$ in $P_0$.
\end{Lem}
\begin{Rem}\label{RM-hopf}
We observe that if the function $c_1$ in Lemma \ref{hopf0} is identically zero, then we can drop the assumption $U(P_0)=0$.
\end{Rem}
\begin{proof}
By assumption, there exist a point $z_0=(\hat x_1,...,\hat x_n,\hat y_1,...,\hat y_n,\hat t,\hat{\lambda})$ and a radius $R>0$ such that the
  ball $ \mathcal{B}(z_0,R)\subset \dD$ and $P_0\in \partial \mathcal{B}(z_0,R)$.

We consider the function
 \begin{align}
\psi=Ue^{-K(x_{1}-\hat{x}_{1})^{2}}, \;\text{for}\; K>0.
 \end{align}
An easy computation yields
 \begin{align*}
 \frac{\partial^{2}\psi}{\partial x_{1}^{2}}=e^{-K(x_{1}-\hat{x}_{1})^{2}}
[4K^{2}(x_{1}-\hat{x}_{1})^{2}U-2K U-4K(x_{1}-\hat{x}_{1})
\frac{\partial  U}{\partial x_{1}}+\frac{\partial^{2}U}{\partial x_{1}^{2}}],
\end{align*}
and
 \begin{align*}
 \frac{\partial^{2}\psi}{\partial x_{j}^{2}}&=e^{-K(x_{1}-\hat{x}_{1})^{2}}
 \frac{\partial^{2}U}{\partial x_{j}^{2}},\quad  \frac{\partial^{2}\psi}{\partial y_{k}^{2}}=e^{-K(x_{1}-\hat{x}_{1})^{2}}
  \frac{\partial^{2}U}{\partial y_{k}^{2}},\\
    \frac{\partial^{2}\psi}{\partial \lambda^{2}}
 &=e^{-K(x_{1}-\hat{x}_{1})^{2}}
  \frac{\partial^{2}U}{\partial \lambda^{2}},
 \end{align*}
 where $j=2,\cdots,n,k=1,\cdots,n.$

 Moreover,
 \begin{align*}
  &\frac{\partial^{2}\psi}{\partial x_{1} \partial t}=e^{-K(x_{1}-\hat{x}_{1})^{2}}
  [-2K(x_{1}-\hat{x}_{1})
 \frac{\partial  U}{\partial t}+\frac{\partial^{2}U}{\partial x_{1}\partial t}],\\
 &\frac{\partial^{2}\psi}{\partial x_{j} \partial t}=e^{-K(x_{1}-\hat{x}_{1})^{2}}
  \frac{\partial^{2}U}{\partial x_{j}\partial t}.
 \end{align*}
 Therefore, we have
\begin{align*}
{\mathcal{L}}\psi+4K(x_{1}-\hat{x}_{1})X_{1}\psi=e^{-K(x_{1}-\hat{x}_{1})^{2}}[-4K^{2}
(x_{1}-\hat{x}_{1})^{2}U+{\mathcal{L}}U-2K U]
\end{align*}
and hence for $K$ sufficiently large, we deduce
$$-{\mathcal{L}}\psi-4K(x_{1}-\hat{x}_{1})X_{1}\psi\geq 0.$$

We introduce now the function $\phi=e^{-\alpha R^{2}}-e^{-\alpha\rho^{2}}$, where $\rho=d_{\widehat{\bH}}(z,z_0)$,
and $0<\rho<R$.
Since $\phi$ depends only on the distance from $z_0$, and $\mathcal L$ and $X_1$ are invariant with respect to the group action in $\widehat{\bH}^{n}$, we can use formula \eqref{radial} where now $\rho=d_{\widehat{\bH}}(z_0^{-1}\circ z,0)$, and the factor $\frac{r^2}{\rho^2}$ is replaced by the function $G(z_0^{-1}\circ z)$, where
\begin{equation}\label{G}
G(x_1,...,x_n,y_1,...,y_n,t,\lambda):=\frac{\sum_{j=1}^n(x_j^2+y_j^2)+\lambda^2}{[(\sum_{j=1}^n(x_j^2+y_j^2)+\lambda^2)^2 + t^2]^{\frac{1}{2}}}.
\end{equation}

Choosing $\alpha$ sufficiently large, we have

\begin{align*}
&-\mathcal L\phi-4K(x_1-\hat{x_1})X_1\phi\\
&\hspace{1em}=\left[G(z_0^{-1}\circ z)(4\alpha^{2}\rho^{2}-2(Q+1)\alpha)\right.\\
&\hspace{1.5em}\left.
-8K\alpha(x_{1}-\hat{x}_{1})\rho X_{1}\rho\right]e^{-\alpha\rho^{2}}\ge0,
\end{align*}

\noindent
Let  ${\mathcal{A}}:= \mathcal{B}(z_0,R)\setminus \mathcal{B}(z_0,R_1)$ for $0<R_1<R$. For $\varepsilon$ small enough
 \[
 \psi(z)+\varepsilon \phi(z)\ge 0\quad \text{in}\; \partial{\mathcal{A}}:= \partial  \mathcal{B}(z_0,R)\cup \partial \mathcal{B}(z_0,R_1).
 \]
Then, by Proposition \ref{prop-max0} we obtain that $\psi(z)+\varepsilon
\phi(z)\ge 0$ in ${\mathcal{A}}$.

\noindent
Therefore, using that $\psi(P_0)=\phi(P_0)=0$, we deduce that for $s$ small
\begin{align*}
 \psi(P_{0})-\psi(P_{0}-s\nu)+\varepsilon(\phi(P_{0})- \phi(P_{0}-s\nu))\leq 0.
\end{align*}
Using that $\phi$ is strictly increasing in $\rho$, we deduce
\[
 \lim\limits_{s\rightarrow 0}\frac{\psi(P_{0})-\psi(P_{0}-s\nu)}{s}<0,
\]
which, in turn, implies that
\[
 \lim\limits_{s\rightarrow 0}\frac{U(P_{0}-s\nu)-U(P_{0})}{s}<0.
\]
\end{proof}

\medskip

For any $\Omega \subset \bH^{n}$, we denote by $\mathcal C$ the infinite cylinder
\begin{align*}
\dC=\Omega \times (0,+\infty).
\end{align*}

Before proving our second Hopf's Lemma,
let us recall the notion of interior ball condition in the Heisenberg group.
\begin{Def}
Let $\Omega\subset \bH^{n}$. We say that $\Omega$ satisfies the interior Heisenberg ball condition at $\xi \in \partial \Omega$ if there exist a constant $R>0$ and a point $\xi_0 \in \Omega$, such that the Heisenberg ball $B_{\bH}(\xi_0,R)\subset \Omega$ and $\xi \in \partial B_{\bH}(\xi_0,R)$.
\end{Def}

\begin{Lem}\label{hopf}
Let $\Omega\subset \CH$ satisfy the interior Heisenberg ball condition at the point $P\in \partial \Omega$ and
let $U\in C^2(\dC)\cap C^1(\overline{\dC})$, be a nonnegative solution of
\begin{equation}\label{eq-hopf}
\begin{cases}
-\mathcal L U \geq c_{1}(z)U &\mbox{in}\;\;\dC,\\
-\partial_\lambda U\geq c_{2}(\xi)U& \mbox{on}\;\; \Omega,
\end{cases}
\end{equation}
with $c_{1},c_{2}\;\in L^\infty(\dC)$ and $c_{1}$ is nonnegative. Suppose that $U((P,0))=0$ and $U$ is not identically null.

Then
\begin{equation}\label{normal-der}
\partial_\nu U(P,0)<0,
\end{equation}
where $\nu$ is the outer normal to $\partial \Omega$ in $P$.
\end{Lem}
\begin{proof}
We follow the proof of Lemma 2.4 in \cite{Chipot}.
%

By the strong maximum principle and by Lemma \ref{hopf0}, we have that
\begin{equation}\label{strong}
U>0\quad \mbox{on}\;\;\dC\cup \Omega.
\end{equation}
Indeed, the strong maximum principle ensures that $U>0$ in $\dC$, moreover $U$ cannot vanish at a point in $\Omega$, otherwise at this point the Neumann condition would be violated by Lemma \ref{hopf0}.

We start by proving the lemma in the case $c_{1}(z)=c_{2}(\xi)\equiv 0$.

Since $\Omega$ satisfies the interior Heisenberg ball condition at $P$, there exist $z_0\in \Omega\times \{0\}$ and $R>0$, such that the ball $\mathcal{B}^+(z_0,R)$ is contained in the cylinder $\dC$ and $(\overline{\partial \dC\cap \{\lambda>0\}})\cap\partial \mathcal{B}^+(z_0,R)=\{(P,0)\}.$
We consider the set
$${\mathcal{A}}=\left(\mathcal{B}^+(z_0,R)\setminus \overline{\mathcal{B}^+(z_0,R/2)}\right)\cap \{\lambda >0\}.$$
We observe that $\{(P,0)\}=\partial {\mathcal{A}}\cap \overline{\partial \dC\cap \{\lambda>0\}}$.

For $z\in {\mathcal{A}}$ we consider the function $\eta(z)=e^{-\alpha \rho^2}-e^{-\alpha R^2}$, where $\rho=d_{\widehat{\bH}}(z,z_0)$.
Writing $\mathcal L$ in radial coordinates as in \eqref{radial}, we have that
$$\mathcal L \eta (\rho)=G(z_0^{-1}\circ z)\left( 4\alpha^2\rho^2-2(Q+1)\alpha\right) e^{-\alpha \rho^2},$$
where $G$ is defined as in \eqref{G}.

Therefore, for $\alpha$ sufficiently large, we have that%
\begin{equation}\label{eta}
-\mathcal L \eta \leq 0.
\end{equation}
By \eqref{strong} we deduce that $U>0$ on $\partial \mathcal{B}^+(z_0,R/2)\cap \{\lambda \geq 0\}$.
 Hence, we may choose $\varepsilon>0$ such that
$$
U- \varepsilon \eta \geq 0\quad \mbox{on}\;\;\partial \mathcal{B}^+(z_0,R/2)\cap \{\lambda \geq 0\}.
$$
\textbf{Claim:} $U-\varepsilon \eta\geq 0$ in ${\mathcal{A}}$.

Indeed, using \eqref{eta}, we deduce that $-\mathcal L(U-\varepsilon \eta)\geq 0$
in ${\mathcal{A}}$. Hence,
by the maximum principle,
we have that the minimum of $U-\varepsilon \eta$
is attained only on $\partial {\mathcal{A}}$ (unless $U-\varepsilon \eta$ is constant). Now, on one side we have that
$$U-\varepsilon \eta \geq 0 \quad \mbox{on}\;\;\partial {\mathcal{A}}\cap \{\lambda>0\}.$$
On the other side, since $\partial_\lambda \eta =0$ on $\{\lambda=0\}$, we deduce that $-\partial_\lambda(U-\varepsilon \eta)\geq 0$ on $\partial \mathcal{A} \cap \{\lambda=0\}$.

\noindent
Thus, using Lemma \ref{hopf0}, we conclude that the minimum of $U-\varepsilon \eta$ cannot be achieved on $\left(\mathcal{B}^{+}(z_0,R)\setminus \overline{\mathcal{B}^+(z_0,R)/2)}\right)\cap \{\lambda=0\}$. This reaches the claim.

\noindent
Finally, since $(U-\varepsilon \eta)((P,0))=0$, we deduce that $\partial_{\nu}(U-\varepsilon \eta)((P,0))\leq 0$, which in turn implies that $\partial_{\nu} U((P,0))<0$ using that $\partial_{\nu}\eta((P,0)) <0$. This concludes the case $c_{1}(z)=c_{2}(\xi)\equiv 0$.

In the general case, we introduce the function $v=e^{-\beta \lambda} U$ and we compute
\begin{eqnarray*}
-\mathcal L v &=&-e^{-\beta \lambda}\mathcal L U +2\beta e^{-\beta \lambda}U_\lambda -\beta^2 v\\
&\geq& c_1(z)-\beta^2 v+ 2\beta e^{-\beta \lambda}(\beta e^{\beta\lambda} v + e^{\beta \lambda}v_\lambda)\\
&=&c_1(z)+\beta^2 v +2\beta v_\lambda.\end{eqnarray*}
Therefore, we have
$$-\mathcal L v-2\beta v_\lambda\geq c_1(z)+\beta^2 v\geq 0.$$
Moreover, for $\beta$ large enough
$$-\partial_\lambda v\geq (\beta+c_2(z))v\geq 0.$$
We can apply the first part of the proof to the function $v$, noting that the same argument works when the operator $\mathcal L$ is replaced by $\mathcal L +2\beta \partial_\lambda$.

\end{proof}

\section{Proof of Theorems \ref{main} and \ref{main2}}
\setcounter{equation}{0}

%
In this last section we give the proof of our Liouville-type result.

We consider a solution $U=U(|(x,y)|,t,\lambda)$ of
\begin{equation}\label{eqn-d}
\begin{cases}
{\mathcal{L}}U=0 &  \text{in}\;  \widehat{\bH}^{n}_{+}:=\CH\times \R^+, \\
   -\partial_{\lambda} U=U^{p} & \text{on}\;  \CH=\partial \widehat{\bH}^{n}_{+}, \\
U>0 &\text{in}\;  \widehat{\bH}^{n}_{+}.
\end{cases}
\end{equation}

First of all, we perform the CR inversion of $U$.
For $z=(x,y,t,\lambda)\in \widehat{\bH}^{n}$, let
\[
w(z)=\frac{1}{\rho^{Q-1}}U({\widetilde z}),
\]
where $\widetilde z=\frac{1}{\rho^{4}}\left(xt+yr^2,yt-xr^2,-t,\lambda \rho^2\right)$,
 $r=\big(\sum_{j=1}^n (x_j^2+y_j^2)+\lambda^2\big)^{\frac{1}{2}} $ and $\rho(z)=(r^{4}+t^2)^{\frac{1}{4}}=
 d_{\widehat{\bH}}(z,0)$.
We have seen in Lemma \ref{CR-eq} that $w$ satisfies
\begin{equation}\label{eqn-depend2}
\begin{cases}
{\mathcal{L}}w=0 &  \text{in}\; \widehat{\bH}^{n}_{+} \setminus \{0\}, \\
   -\partial_{\lambda} w=\rho^{p(Q-1)-(Q+1)} w^{p} & \text{on}\;  \CH\times \{\lambda=0\}\setminus \{0\}.
\end{cases}
\end{equation}
Observe that the function  $w$ could be singular at the origin and it satisfies $\lim_{\rho\rightarrow \infty} \rho^{Q-1} w(z)=U(0)$,
hence
\begin{equation}\label{tildeC}
0<w(z)\leq \frac{\widetilde C}{\rho(z)^{Q-1}},\quad \mbox{for}\:\:\rho(z)\geq 1,
\end{equation}
for some positive constant $\widetilde C$.
We start now applying the moving plane method.
We will move a hyperplane orthogonal to the $t$-direction and use the $H$-reflection.
More precisely, for any $\mu\leq 0$, let

$T_{\mu}=\{z\in \widehat{\bH}^{n}_{+}\mid
t=\mu\}$, and $\Sigma_{\mu}=\{z\in \widehat{\bH}^{n}_{+}\mid t<\mu\}$.
For $z\in \Sigma_{\mu}$, we define
$z_{\mu}=(y,x,2\mu-t, \lambda)$. To avoid the singular point,
we consider
\[
\widetilde{\Sigma}_{\mu}=\Sigma_{\mu}\setminus\{e_{\mu}\},
\]
where $e_\mu=(0,0,2\mu,0)$ is the reflection of the origin.
We recall that, as shown in the proof of Lemma \ref{CR-eq}, if $U$ depends only on $(r_0,t,\lambda)$, then so does $w$.

Let now
\begin{eqnarray*}
w_\mu(z)&=&w_\mu(|(x,y)|,t,\lambda):=w(|(x,y)|,2\mu-t,\lambda)\\
&=&w(y,x,2\mu-t,\lambda)=w(z_{\mu}),
\end{eqnarray*}
and
\[
W_{\mu}(z):=w_{\mu}(z)-w(z)=w(z_{\mu})-w(z),\quad z\in \Sigma_{\mu}.
\]
By using the invariance of the operator under the CR transform as in Lemma \ref{CR-eq} and
the fact that $\rho(z_{\mu})\le \rho(z)$, we deduce that
\[
\begin{cases}
{\mathcal{L}}  W_{\mu} =0 &  \text{in}\; \widehat{\bH}^{n}_{+}\setminus \{0,e_\mu\}, \\
 -\partial_{\lambda}W_{\mu}\ge c(z,\mu)W_{\mu} &  \text{on}\; \CH\setminus \{0,e_\mu\},
\end{cases}
\]
where $c(z,\mu)=\frac{p\Psi_{\mu}^{p-1}}{\rho^{(Q+1)-p(Q-1)}}$ and $\Psi_{\mu}(z)$ is between $w(z)$ and $w_{\mu}(z)$.
By the definition of $w_{\mu}$ and $w$, we have that $c(z,\mu)\approx C/\rho^{2}$ at infinity.

\medskip

We now define the function:
$$h_{0}=\rho(z+\beta_{2}e_{2n+2})^{-\beta_{1}},$$
where $e_{2n+2}=(0, \cdots, 0, 1)$, $\beta_2>0$ and
\begin{equation}\label{beta}
0<\beta_{1}<Q-1,\quad p\widetilde C^{p-1}<\beta_1\beta_2,
\end{equation}
where $\widetilde C$ is the constant in \eqref{tildeC}.

We consider
$$\widetilde{W}_{\mu}=W_\mu/h_{0}\quad \mbox{in} \;\;\Sigma_{\mu}.$$

The following lemma will let us start moving the hyperplane $T_{\mu}$ from $-\infty$.

\begin{Prop}\label{Prop-incia-mp}
 Assume that  $w\in C^{2}(\widehat{\bH}^{n}_{+})\cap C^{1}(\overline{\widehat{\bH}^{n}_{+}})\setminus \{0\}$
 satisfies \eqref{eqn-depend2}.
Then
{\rm(i)} For $\mu<0$ with $|\mu|$ large enough, if $\inf_{\Sigma_\mu} \widetilde{W}_\mu <0$, then the infimum is attained at some point   $z_0 \in \overline{\Sigma_\mu}\setminus\{e_\mu\}$.
{\rm(ii)} For any $\mu_0$, there exists an $R_1>0$ such that whenever $\inf_{\Sigma_\mu} \widetilde{W}_\mu$  is attained at $z_0 \in \overline{\Sigma_\mu}\setminus\{e_\mu\}$ with $\widetilde{W}_\mu(z_0) <0$ and $\mu \leq \mu_0$, then $\rho(z_0)=d_{\widehat{\bH}^n}(z_{0},0)\leq R_1$.

\end{Prop}
\begin{proof} We follow the proof of Proposition 6.3 in \cite{Chipot}.

We first observe that by the maximum principle and Hopf's Lemma (see Lemma \ref{hopf0})
$$\min\{w(z):z\in \partial{\mathcal{B}^{+}(0,1)}\cap \{\lambda>0\}\}=d>0.$$

We define for $z\in {\mathcal{A}}_{\varepsilon}:=
\overline{\mathcal{B}^{+}(0,1)}\setminus \mathcal{B}^{+}(0,\varepsilon)$
the function
\[\phi_{\varepsilon}(z)=d\frac{\rho(z)^{1-Q}-\varepsilon^{1-Q}}{1-\varepsilon^{Q-1}}.\]
By Lemma \ref{foundamental} $\phi_{\varepsilon}$ is harmonic in $\mathcal A_\varepsilon$ and satisfies $\frac{\partial\phi_{\varepsilon}}{\partial \lambda} =0$ on $\mathcal A_\varepsilon \cap \{\lambda=0\}$.

By the maximum principle and Hopf's Lemma we have that $w(z)\geq \phi_\varepsilon(z)$ for $z\in {\mathcal{A}}_{\varepsilon}$. Since for every $z$ we have $\lim_{\varepsilon \rightarrow 0}\phi_\varepsilon(z)=d\rho(z)^{1-Q}\geq d$, we deduce that
$$w(z)\geq d \quad \mbox{on}\:\:\overline{\mathcal{B}^{+}(0,1)}\setminus \{0\}.$$
Since $w(z)\rightarrow 0$ as $\rho(z)\rightarrow \infty$, it follows that for $\mu$ large in absolute value, we have $w\leq d$ in $\mathcal B^+(z_\mu,1)$. For such $\mu$ we have clearly $W_\mu \geq 0$ on $\mathcal B^+(z_\mu,1)$, and therefore
\begin{equation}\label{W>0}
\widetilde{W}_\mu \geq 0 \quad \mbox{on}\;\;\mathcal B^+(z_\mu,1).
\end{equation}
It follows that, for $\mu<0$ large in absolute value,
$$\inf_{\Sigma_\mu}\widetilde{W}_\mu <0\quad \mbox{implies}\quad \inf_{\Sigma_\mu} \widetilde{W}_\mu = \inf_{\Sigma_\mu\setminus \mathcal B^+(z_\mu,1)}\widetilde{W}_\mu.$$
We conclude the proof of i) observing that $\widetilde{W}_\mu(z) \rightarrow 0$ as $\rho(z)\rightarrow \infty$.

 To prove (ii),  suppose that $z_{0}$ is a minimum point of $\widetilde{W}_{\mu}$ such that
$\widetilde{W}_{\mu}(z_0)<0$. We want to show that  $\rho(z_{0})$ cannot be too large.

 By the definition of $\widetilde{W}_\mu$, a direct calculation gives
      \[
\begin{cases}
-{\mathcal{L}} \widetilde{W}_{\mu}=\frac{ {\mathcal{L}}( h_{0})}{h_{0}} \widetilde{W}_{\mu}
 +2(\frac{\nabla_{\bH} h_{0}}{h_{0}}\cdot \nabla_{\bH} \widetilde{W}_{\mu}+\frac{\partial_\lambda h_0}{h_0} \partial_\lambda \widetilde{W}_{\mu}+4\lambda^2\frac{\partial_t h_0}{h_0} \partial_t \widetilde{W}_{\mu})\;\;\;  \text{in}\; \Sigma_{\mu}\setminus\{e_\mu\}, \\
   -\partial_{\lambda} \widetilde{W}_{\mu}\geq \left(c(z,\mu)+ \frac{\partial_{\lambda} h_{0}}{h_{0}} \right)\widetilde{W}_{\mu} \;\;\;\text{on}\;   \partial\Sigma_{\mu}\cap\{\lambda=0\}\setminus\{z_{\mu}\},
\end{cases}
\]
where $c(z,\mu)=p\rho^{p(Q-1)-(Q+1)}
   \Psi_{\mu}^{p-1}$ and $\Psi_{\mu}(z)$ is between $w_{\mu}(z)$ and $w(z)$.
Since
$\frac{{\mathcal{L}} (h_{0})}{h_{0}}
=\beta_{1}\big(\beta_{1}-(Q-1)\big)\rho(z+\beta_{2}e_{2n+2})^{-2}<0$, using the maximum principle we deduce
that $z_{0}$ does not belong to the interior of  $\Sigma_{\mu}$.

Assume now that $z_0\in \partial\Sigma_\mu\cap\{\lambda=0\}\setminus\{z_\mu\}$. As above we conclude that $\Psi_\mu(z_0)\leq w(z_0)\leq \frac{\widetilde C}{\rho(z_0)^{q-1}}$, and hence that $c(z_0,\mu)\leq \frac{p \widetilde C^{p-1}}{\rho(z_0)^2}$.
Since $\frac{\partial_\lambda h_0}{h_0}=-\beta_1\beta_2 \rho(z+\beta_2 e_{2n+2})^{-2}$, using assumptions \eqref{beta},  we would deduce that
$\partial_{\lambda}\widetilde{W}_{\mu}(z_0)<0$ if $\rho(z_0)$ were large enough. This is inconsistent
   with the fact $z_{0}$ is a minimum and concludes the proof of ii).
\end{proof}

\medskip

\begin{proof}[Proof of Theorem \ref{main2}]
By Proposition \ref{Prop-incia-mp}, we deduce that for $\mu$ negative and large in absolute value we have that $\widetilde{W}_\mu \geq 0$ and hence $W_\mu \geq 0$ in $\Sigma_\mu$.
Let us define
$\mu_{0}=\sup \{ \mu<0\mid
 W_{\sigma}\ge 0\;\;\mbox{on}\:\:\Sigma_\sigma \setminus e_\sigma\:\:\mbox{for all}\:\:\sigma < \mu \}$.
 We only need to prove that
$\mu_{0}=0$. Suppose that $\mu_{0}\neq 0$ by contradiction.
By continuity, $W_{\mu_0} \geq 0$ in $\Sigma_{\mu_0}$. By the maximum principle (and Hopf's Lemma), we deduce that $W_{\mu_0}\equiv 0$ in $\Sigma_{\mu_0}$ or
\begin{equation}\label{w>0}
W_{\mu_0}>0 \quad \mbox{on}\quad \Sigma_{\mu_0}\cup \left(\partial \Sigma_{\mu_0} \cap \{\lambda=0\}\cap \{t<\mu_0\}\right)\setminus\{e_{\mu_0}\}.
\end{equation}
If $W_{\mu_0} \equiv 0$, then $w$ would be even in the $t$ variable with respect to $t=\mu_0<0$ and this would contradict the Neumann condition satisfied by $w$
(we remind that $\rho(z_{\mu_0})<\rho(z)$ since we are assuming $\mu_0<0$), hence $W_{\mu_0} \equiv 0$ is impossible and therefore \eqref{w>0} holds.

By the definition of $\mu_0$
there exists $\mu_{k}\rightarrow\mu_{0}$, $\mu_{0}<\mu_{k}<0$ such that
$\inf_{{\Sigma_{\mu_{k}}}}W_{\mu_{k}}<0$.

We observe that for some positive  $b_{1}$:
  \[\min\left\{W_{\mu_0}(z)\mid\,z \in \partial \mathcal{B}^+(e_{\mu_0},|\mu_0|/2)\cap \widehat{\bH}^{n}_{+} \right\}=b_{1}.\]
   From this fact, using a similar argument to the one of point i) in Proposition \ref{Prop-incia-mp}, we deduce that
$$W_{\mu_0}\geq b_{1} \quad \mbox{in}\;\;  \overline{\mathcal{B}^+(e_{\mu_0}, |\mu_0|/2)}\setminus\{e_{\mu_0}\}.$$
Therefore, we have that
$$\lim_{k\rightarrow \infty}\inf\left\{W_{\mu_k}(z)\mid\,z\in \mathcal{B}^+(e_{\mu_k},|\mu_0|/2)\setminus \{e_{\mu_k}\}\right\}\geq b_{1}.$$
Using this bound and the fact that $W_{\mu_k}(z)\rightarrow 0$ as $\rho(z)\rightarrow \infty$, we deduce that for $k$ large enough, the negative infimum of $W_{\mu_k}$ is attained at some point $z_k \in \overline {\Sigma_{\mu_k}}\setminus \mathcal{B}^+(e_{\mu_k},|\mu_0|/2).$

By Proposition \ref{Prop-incia-mp} we know that the sequence $\{z_{k}\}$ is bounded and therefore, after passing to a subsequence, we may assume that $z_{k}\rightarrow z_{0}$. By \eqref{w>0} we have that $W_{\mu_0}(z_0)=0$ and $z_0 \in \partial \Sigma_{\mu_0}\cap\{t=\mu_0\}$.

If $z_k\in \Sigma_{\mu_k}\cap \{\lambda >0\}$ for an infinite number of $k$, then $\nabla W_{\mu_k}(z_k)=0$, and therefore, by continuity
\begin{equation}\label{grad=0}
\nabla  W_{\mu_0}(z_0)=0.
\end{equation}
If $z_0 \in \partial \Sigma_{\mu_0}\cap \widehat{\bH}^{n}_{+}$, then by  Lemma \ref{hopf0}, we have that $\frac{\partial w}{\partial t}(z_0) <0$, which gives a contradiction. Analogously, using Lemma \ref{hopf}, we get a contradiction if we assume that $z_0 \in \partial \Sigma_{\mu_0}\cap \{\lambda=0\} \cap \{t=\mu_0\}.$

In the case in which $z_k \in \partial \Sigma_{\mu_k}\cap \{\lambda=0\}\cap \{t<\mu_k\}$, we still have that the derivatives of $W_{\mu_k}$ at $z_k$ in all directions except the $\lambda$ direction vanish. Passing to the limit and arguing as above, we get a contradiction. Hence we have established that $\mu_0=0$.
This implies that $v$ is even in $t$, but since the origin $0$ on the $t$-axes is arbitrary, we can perform the CR transform with respect to any point and then we conclude that $w$ is  constant in the direction  $t$.

This shows that $U$ is actually a solution of the following problem
\begin{equation}\label{classic}
\begin{cases}
{\Delta}  U =0 &  \text{in}\; \R_+^{2n+1}, \\
 -\partial_{\lambda}U= U^{p} &  \text{on}\; \R^{2n}.
\end{cases}
\end{equation}
Since $\frac{Q+1}{Q-1}=\frac{2n+3}{2n+1} <\frac{2n+1}{2n-1}$, we conclude the proof by using the standard Liouville type theorem for problem \eqref{classic} (see \cite{Chipot,LZ}).

\end{proof}

The following lemma allows us to deduce Theorem \ref{main} by Theorem \ref{main2}.
\medskip
\begin{Lem}\label{lem-rad} Let $u\in C^{2}(\bH^{n})$ be cylindrically  symmetric and positive (respectively nonnegative). Then
the corresponding  solution $U$ of the extension problem \eqref{eq-ext} is cylindrically simmetric, i.e. $U=U(r_0,t,\lambda)$
with $r_{0}=\sqrt{x^{2}+y^{2}}$, $U$ is positive (respectively nonnegative) and moreover
$U\in C^{2}(\widehat{\bH}_{+}^{n})\cap C^{1}(\widehat{\bH}_{+}^{n})$.
\end{Lem}
\begin{proof}
The main tool in the proof of the Lemma relies on the construction of the extension $U=\mathcal E_{1/2}u$ by using the Fourier transform in $\mathbb H^n$ (we refer for details to Section 5 in \cite{FGMT} and to \cite{BG}).
Let us recall  the
 Fourier transform of a smooth function $u(\zeta,t)$, $(\zeta,t) \in \bH^{n}$,
\begin{align}
{\mathcal{F}}(u)(\mu)=\int_{\mathbb H^{n}}u(\zeta,t)\pi_{\zeta,t}^{\mu}\,d\zeta dt,
\end{align}
where $\pi_{\zeta,t}^{\mu}$ denotes the irreducible representation
\begin{align*}
\pi_{\zeta ,t}^{\mu}\Psi(\xi)=
\left\{
\begin{array}{ll}
\Psi(\xi-\bar{\zeta })e^{i\mu t+2\mu(\xi\cdot \zeta -|\zeta |^{2}/2)},&\mu>0,\\
 \Psi(\xi+\zeta )e^{i\mu t+2\mu(\xi\cdot \bar{\zeta }-|\zeta |^{2}/2)},&\mu<0,
\end{array}
\right.
 \end{align*}
for a holomorphic function  $\Psi(\xi), \xi\in\bH^{n}$  in the Bargmann space with
orthogonal basis
$$\Psi_{\alpha,\mu}(\xi)=\frac{(\sqrt{2|\mu|}\xi)^{\alpha}}{\sqrt{\alpha !}}.$$

The inversion formula is given by
\begin{align}
u(\zeta,t)=
\frac{2^{n-1}}{\pi^{n+1}}\int_{\R}\text{tr}\;\pi^{*,\mu}_{\zeta,t}
{\mathcal{F}}(u)(\mu)\,|\mu|^{n}d\mu
\end{align}
where $ \pi_{\zeta ,t}^{*,\mu}=\pi_{(\zeta ,t)^{-1}}^{\mu}$ is the adjoint operator of $\pi_{\zeta ,t}^{\mu}$.

The extension operator $\mathcal E_{1/2}$   maps a function $u$ on the Heisenberg group  to a function $U=\mathcal E_{1/2} u$ on  $\bH^n\times (0,\infty)$.
For  every $q>0$, every multi-index vector $\alpha$ and $\mu\in\R$,
 $U (\cdot,q)=\mathcal E_{1/2} u(\cdot,q)$  is  implicitly given
  through the Fourier multiplier
\begin{equation}\label{eqnq20}
\widehat{\mathcal E_{1/2} u(\cdot,q)}_\alpha(\mu)=\widehat{U}_\alpha(\mu,q) = \phi_\alpha(2|\mu| q) \ \hat u_\alpha(\mu),
\end{equation}
where $\hat{u}_\alpha(\mu)={\mathcal{F}}(u)(\mu)\Psi_{\alpha,\mu}$,
\begin{align*}
\phi_\alpha(y) = \frac{\Gamma\big( \frac{\frac{3}{2}+n+2|\alpha|}{2}\big)}{\Gamma(\gamma)}
 \ e^{-y/2} \ y^{\frac{1}{2}} \ V\Big( \frac{\frac{3}{2}+n+2|\alpha|}{2},\frac{3}{2},y\Big),
\end{align*}
and $V(a,b,y)$ is the solution of Kummer's equation
\begin{align*}
V(a,b,y) = \frac{1}{\Gamma(a)} \int_0^\infty e^{-\tau y} \tau^{a-1} (1+\tau)^{b-a-1} \,d\tau \,.
\end{align*}
Note that the function $q\mapsto \widehat{\mathcal E_{1/2}u(\cdot\:,q)_\alpha(\mu)}$ solves the equation
\begin{align*}
\big( q\partial_{qq} + \frac{1}{2}\partial_{q}- \mu^2 q -
|\mu|(n+2|\alpha|) \big) \widehat{\mathcal E_{1/2} u(\cdot,q)}_\alpha(\mu) = 0 \,,
\end{align*}
 and therefore $U=\mathcal E_{1/2}u$ satisfies
 \begin{align}\label{ext-q}
\left( q\partial_{qq} +  \frac{1}{2}\partial_{q}+ q\partial_{tt} + \frac{1}{2}
\Delta_b \right)U= 0.
\end{align}
We observe that the change of variable $q=\frac{\lambda^2}{2}$ transforms equation \eqref{ext-q} above into the extension \eqref{extension}.

For every multi-index vector $\alpha$ and $\mu\in\R$, we have
\begin{align*}
&U_{\alpha}(\tilde{\zeta},\tilde{t},q)=
\frac{2^{n-1}}{\pi^{n+1}}\int_{\R}\Big[\text{tr}\;\pi^{*,\mu}_{\tilde{\zeta},\tilde{t}}
 \phi_\alpha(2|\mu| q) \ \hat u_\alpha(\mu)\,|\mu|^{n}\Big]d\mu\\
&\hspace{1em} =\frac{2^{n-1}}{\pi^{n+1}}\int_{\R}\Big[\text{tr}\;\pi^{*\mu}_{\tilde{\zeta},\tilde{t}}
 \phi_\alpha(2|\mu| q) \ \big(\int_{\mathbb H^{n}}u(\zeta,t)\pi_{\zeta,t}^{\mu}\,d\zeta dt \Psi_{\alpha,\mu} \big)
 \,|\mu|^{n}\Big]d\mu\\
 &\hspace{1em} =\frac{2^{n-1}}{\pi^{n+1}}\int_{\R}\int_{\mathbb H^{n}}\Big[ \text{tr}\;
 \phi_\alpha(2|\mu| q) \ \big(u(\zeta,t)\pi^{\mu}_{(\tilde{\zeta},\tilde{t})^{-1}} \pi_{\zeta,t}^{\mu} \Psi_{\alpha,\mu} \big)
 \,|\mu|^{n}\Big] \,d\zeta dtd\mu.
\end{align*}
We observe that
if $u$ is cylindrically symmetric with
respect to $\zeta$, that is, for the rotation ${\mathcal{R}}_{x,y}$,
\[u(\zeta,t)=u({\mathcal{R}}_{x,y}(\zeta),t),\]
then the extension $U_{\alpha}(\zeta,t,q)$ of  $\widehat{\mathcal E_{1/2} u(\cdot,q)}_\alpha$
 is cylindrically symmetric with
respect to $\zeta=(x,y)$ for each $q$.
 That is,  \[U(\zeta,t,q)=U({\mathcal{R}}_{x,y}(\zeta),t,q).\]
 After the change of variables $q=\lambda^2/2$, we deduce that
$U(\zeta,t,\lambda)$ is  cylindrically symmetric with
respect to $\zeta=(x,y)$ for each $\lambda$.

\end{proof}

\bigskip

\begin{proof}[Proof of Theorem \ref{main}] By Lemma \ref{lem-rad},
we see that if $u$ is a cylindrical function, that is $u=u(|(x,y)|,t)$, then its extension $U$ satisfying \eqref{extension} is also cylindrical in the all halfspace $\mathbb H^n\times \R^+$, in the sense that $U=U(|(x,y)|,t,\lambda)$.
Using this fact, the conclusion follows as a corollary of Theorem \ref{main2}.

\end{proof}

\bigskip
\noindent {\bf Acknowledgements:}  Both authors were supported  by
 Spain Government grant MTM2011-27739-C04-01;
 J.T. was supported by Chile  Government grant Fondecyt 1120105, USM 121402;
  CMM in Universidad de Chile.


\end{document}